\documentclass[11pt,reqno]{amsart}
\usepackage{graphicx}
\usepackage{float}
\usepackage[caption = false]{subfig}
\usepackage{amsmath}
\usepackage{enumerate}
\usepackage{mathtools}
\usepackage[english]{babel}
\usepackage{amsfonts,amssymb}
\usepackage{mathrsfs}
\usepackage[utf8x]{inputenc}\usepackage[english]{babel}
\usepackage{soul}
\usepackage[all]{xy,xypic}
\usepackage{cite}
\numberwithin{equation}{section}
\newtheorem{thm}{Theorem}[section]
\newtheorem{cor}[thm]{Corollary}
\newtheorem{lem}[thm]{Lemma}

\theoremstyle{definition}
\newtheorem{defn}[thm]{Definition}
\theoremstyle{remark}
\newtheorem{rem}[thm]{Remark}
\newtheorem{example}{Example}
\newtheorem{conjecture}{Conjecture}
\numberwithin{equation}{section}
\DeclareMathOperator{\RE}{Re}
\DeclareMathOperator{\IM}{Im}

\begin{document}
	
	\title[\tiny{On a Special Type of Ma-Minda Function}]{On a Special Type of Ma-Minda Function}

	\author[S. Sivaprasad Kumar]{S. Sivaprasad Kumar}
	\address{Department of Applied Mathematics, Delhi Technological University, Delhi--110042, India}
	\email{spkumar@dce.ac.in}

	\author[S. Banga]{Shagun Banga}
	\address{Department of Applied Mathematics, Delhi Technological University, Delhi--110042, India}
	\email{shagun05banga@gmail.com}

	
	\maketitle
	\begin{abstract} This paper deals with a special type of Ma-Minda function, introduced here with many fascinating facts and interesting applications. It is much akin in all aspects but differs by a condition from its Ma-Minda counterpart. Further, we consider the function:~$1-\log(1+z)$, a special Ma-Minda of the type introduced here, to define a subclass of starlike functions in a similar fashion as we do with Ma-Minda function and is studied for establishing inclusion and radius results. Apart from that, we also deal with the majorization and Bloch function norm problems for the same class. In addition, we obtain the bounds of fourth coefficient:~$a_4$ and second Hankel determinant:~$H_2(2)$ for the functions belonging to a newly defined class using convolution, which generalizes many earlier known results and its association with the special type of Ma-Minda function is also pointed out.
	\end{abstract}
	
	\section{Introduction}
	\label{intro}
	Let $\mathcal{A}$ be the set of all normalized analytic functions $f$ of the form $f(z)=z+a_2 z^2+a_3 z^3+\cdots$ defined on the unit disk $\mathbb{D}=\{z:|z| <1\}$.
The subset $\mathcal{S}$ of $\mathcal{A}$ denotes the class of univalent normalized analytic functions.  We say, $f$ is subordinate to $g$, denoted by $f \prec g$, if there exists a Schwarz function $\omega$ with $\omega(0)=0$ and $|\omega(z)|<1$ such that $f(z)=g(\omega(z))$, where $f$ and $g$ are analytic functions. Moreover, if $g$ is univalent in $\mathbb{D}$, $f(z) \prec g(z)$ if and only if $f(0)=g(0)$ and $f(\mathbb{D}) \subseteq g(\mathbb{D})$. Define $g,h \in \mathcal{A}$ as: \begin{equation}\label{conv} g(z)= z+\sum_{n=2}^{\infty} g_n z^n \text{ and } h(z)= z+\sum_{n=2}^{\infty} h_n z^n.\end{equation} The convolution (Hadamard Product) of $g$ and $h$ is given by $(g*h)(z)=z+\sum_{n=2}^{\infty}g_n h_n z^n.$
We now introduce the following primitive class, which specializes in several well-known classes:
\begin{equation*}\mathcal{A}(g,h,\varphi)=\left\{f \in \mathcal{A}:\dfrac{(f*g)(z)}{(f*h)(z)} \prec \varphi(z), \varphi\;\text{is analytic  univalent and }\; \varphi(0)=1\right\}.\end{equation*} 
In  1985, Padmanabhan and Parvatham~\cite{padman}, considered the class $\mathcal{A}(K_a*g_1,K_a*h_1,\varphi)$,  where $K_a(z)=z/(1-z)^a$ $(a \text{ Real})$, $g_1(z):=z/(1-z)^2$ and $h_1(z):=z/(1-z)$  by imposing additional  conditions on  $\varphi$, namely it is convex and $\RE \varphi>0$. Later in the year 1989, Shanmugam~\cite{shan} extended $\mathcal{A}(K_a*g_1,K_a*h_1,\varphi)$ to $\mathcal{A}(g*g_1,g*h_1,\varphi)$ by considering a more general $g$ in place of $K_a(z)$. In 1992, Ma and Minda \cite{p3} tweaked the conditions on $\varphi$, which we shall denote by $\phi$ to introduce their own subclasses of starlike and convex functions, namely
\begin{equation}\label{maclass}
	\mathcal{S}^*(\phi)=\left\{f \in \mathcal{S}: \dfrac{z f'(z)}{f(z)} \prec \phi(z)\right\} \text{ and } \mathcal{C}(\phi)=\left\{f \in \mathcal{S}: 1+\dfrac{z f''(z)}{f'(z)}\prec \phi(z)\right\}.\end{equation} 
The Taylor series expansion of such $\phi(z)$ be of the form:
\begin{equation}\label{phi}
	\phi(z)= 1+B_1 z+B_2 z^2+B_3 z^3+\cdots\text{ } (B_1>0).\end{equation} Note that $\mathcal{S}^*(\phi):=\mathcal{A}(g_1,h_1,\phi)$ and  $\mathcal{C}(\phi):=\mathcal{A}(g_2,g_1,\phi)$ whenever $f$ is univalent and $g_2(z):=(z+z^2)/(1-z)^3.$  
We now classify the Ma-Minda function in the following definition on the basis of its conditions:
\begin{defn} \label{mm} An analytic univalent function $\phi$ with $\phi'(0) >0$, satisfying:
	\begin{description}
		\item[A] $\RE\phi(z)>0 \;\; (z\in \mathbb{D})$
		\item[B] $\phi(\mathbb{D})$ symmetric about the real axis and  starlike with respect to $\phi(0)=1$ 
	\end{description}
	is called a \emph{Ma-Minda function}, we denote the class of all such functions by $\mathscr{M}$.
	If  the condition A  above alone is relaxed, the resulting function, we call it a \emph{non-Ma-Minda of type-A}, the class of all such functions is denoted by $\widetilde{\mathscr{M}}_{\textbf{A}}.$ 
\end{defn}
Recently, the classes given in (\ref{maclass}) were studied extensively for different choices of  $\phi$. Prominently, Aouf et al. \cite{aouf}, studied the class $\mathcal{S}^*(q_c)$, where $q_c=\sqrt{1+c z}$ $(0 <c \leq 1)$,  Robertson \cite{robert} introduced the class of starlike functions of order alpha $(0\leq \alpha \leq 1)$, denoted by $\mathcal{S}^*(\alpha)$ by opting  $\phi(z)$ to be $(1+(1-2\alpha)z)/(1-z)$ and when $\phi(z)=((1+z)/(1-z))^{\eta}$, $\mathcal{S}^*(\phi)$ reduces to the class of strongly starlike functions of order $\eta$, which can be represented in terms of argument as $\mathcal{SS}^*(\eta):=\{f \in \mathcal{A}: |\arg z f'(z)/f(z)| < \eta \pi/2, \text{ } (0 <\eta \leq 1)\}.$ 
Consider the class $k-\mathcal{SP}(\alpha, \beta)$, which is introduced in~\cite{conic}, for $k=1$, it reduces to $1-\mathcal{SP}(\alpha, \beta):=\{f \in \mathcal{A}: z f'(z)/f(z) \prec \phi(z) \},$
where $\phi(z)= \alpha+((2(\alpha-\beta))/\pi^2)(\log((1+\sqrt{\nu(z)})/(1-\sqrt{\nu(z)})))^2$ with $\nu(z)=(z+\rho)/(1+\rho z)$, $\rho =((e^A-1)/(e^A+1))^2$ and $A=\sqrt{(1-\alpha)/(2(\alpha-\beta))}\pi.$

The authors in~\cite{banga,priyanka,journal,exponential} dealt with the radius, inclusion and differential subordination results for the classes involving  $\phi(z)$. Many authors have determined the coefficient bounds for the classes associated with $\phi(z)$ (see \cite{sri,second,prajapat,p9,raza}).In the past, authors considered non-Ma-Minda functions, for instance Kargar et al.~\cite{kargar} and Uralegaddi et al.~\cite{ura} considered functions in $\widetilde{\mathscr{M}}_{\textbf{A}}$ to define their classes. 

We come across the following observations, enlisted below, while examining the geometry of a function defined on $\mathbb{D}$ in general, which are of great use in deriving our results:
\begin{enumerate}
	\item A function with real coefficients is always symmetric with respect to the real axis, but not conversely, for instance:
	\begin{equation*}
		f_1(z)= iz \text{, } f_2(z)=1+iz \text{, } f_3(z)=\dfrac{1+iz}{1-z^2}.\end{equation*}
	The converse holds under special conditions, namely if $f$ is symmetric with respect to the real axis, $f(0)=0$ and $f'(0)$ be some non zero real number, then the function $f$ has real coefficients. In fact the functions $f_1$, $f_2$ and $f_3$ are symmetric with respect to the real axis but do not have real coefficients as $f'_i(0)$ is not a real number for $(i=1,2,3)$.
	\item Let $f(z)$ be an analytic function with real coefficients and $f(0)=0$. Then $f$ is  typically real if and only if its first coefficient is positive. Thus $\phi'(0)>0$ implies $\phi-1$ is typically real whereas the function $\Phi -1$ is non typically real due to $\Phi'(0)<0$.
	\item Geometrically, it is  evident that the real part of a function attains its maximum/ minimum value on the real line if and only if the function is symmetric with respect to the real axis and convex in the direction of imaginary axis.
\end{enumerate}
	
	The  Ma-Minda function $\phi$ is considered as univalent and therefore $\phi'(0)\neq 0$. Since $\phi(\mathbb{D})$ is symmetric about the real axis and if $\phi'(0)$ is any non-zero real number, then $\phi$ has real coefficients. Now to address distortion theorem, Ma-Minda perhaps restricted $\phi'(0)$ to be positive instead of any non-zero real number. However, it has no influence in establishing the coefficient, radius, inclusion, subordination, and other results for the classes $\mathcal{C(\phi)}$ and $\mathcal{S}^*(\phi)$. This very fact, which is under gloom until now, has been brought to daylight in this paper by replacing the condition $\phi'(0)> 0$ with $\phi'(0)< 0$.	Note that $\phi(z)$ and $\Phi(z):=\phi(-z)$ both map unit disk to the same image but different  orientation. 	Thus $\Phi(z)$ differs from its Ma-Minda counterpart by mere a rotation and is therefore non-typically real, but still, image domain invariant and rest all properties are intact. So $\Phi(z)$ can be considered as {\it a special type of Ma-Minda} function. We now premise the above notion in the following definition:
		\begin{defn}
		An analytic univalent function $\Phi$ defined on the unit disk $\mathbb{D}$ is said to be a $\textit{special type of Ma-Minda}$ if $\RE\Phi(\mathbb{D})>0$, $\Phi(\mathbb{D})$ is symmetric with respect to the real axis, starlike with respect to $\Phi(0)=1$ and $\Phi'(0) <0$. Further, it has a power series expansion of the form:
		\begin{equation*}
			\Phi(z)=1+\sum_{n=1}^{\infty} C_n z^n=1+C_1 z+C_2 z^2+\cdots \text{ } (C_1 <0).
		\end{equation*}
		The class of all such  special type of Ma-Minda functions are denoted by $\mathscr{M}^\circ$.
	\end{defn}
	Recently, Altinkaya et al.~\cite{alt} considered a special type of Ma-Minda function $g(z)= {\alpha(1-z)}/{(\alpha -z)}$,  ($\alpha>1)$) to define and study their class $P(\alpha)$. Now the classes $\mathcal{S}^*(\Phi)$ and $\mathcal{C}(\Phi)$ can be defined on the similar lines of $\eqref{maclass}$. We introduce here a special type of Ma-Minda function, given by 
	\begin{equation}\label{log1}
		\psi(z):= 1-\log(1+z)= 1-z+\dfrac{z^2}{2}-\dfrac{z^3}{3}+\cdots,
	\end{equation} which  maps the unit disk onto a parabolic region, see Figure~\ref{fig:incl_rel} for its boundary curve $ \tau $. Although $\phi(\mathbb{D})=\Phi(\mathbb{D})$, at times considering  $\Phi$  is advantageous over its counterpart $\phi$, which is evident from the example $\Phi(z)=1-\log(1+z)$, dealt here. Another such example is $\cos\sqrt{z}$.  Thus the special type of Ma-Minda functions can now be considered in defining Ma-Minda classes for computational convenience as all results are alike except distortion and growth. We now list in Table~1, a few examples of  $\phi \in \mathscr{M}$ and its counter part $\Phi \in \mathscr{M}^{\circ}$:

	\begin{center} \large
	\begin{tabular}{|c|c|}
		\hline	
		$\phi(z)$ &  $\Phi(z)$  \\
		\hline
		$\cos\sqrt{-z}$ & $\cos\sqrt{z}$   \\ $\sqrt{1+z}$ & $\sqrt{1-z}$ \\ $1-\log(1-z)$ &  $1-\log(1+z)$ \\	\hline
	\end{tabular}\\
	{ \small Table 1. Examples of Ma-Minda and its counter part Special type of Ma-Minda functions.}
\end{center}  
	\textit{Distortion and Growth Theorems:}
Let us define the functions in a similar manner as that in \cite{p3}: $d_{\Phi n}(n=1,2,3,\cdots)$ by $d_{\Phi n}(0)=d'_{\Phi n}(0)-1=0$ and
\begin{equation} \label{dist} 1+\dfrac{z d''_{\Phi n}(z)}{d'_{\Phi n}(z)}=\Phi(z^n),\end{equation} which belongs to the class $\mathcal{C}(\Phi)$ and we write $d_{\Phi 1}$ as $d_{\Phi}$. The structural formula of $d'_{\Phi n}$ is given by:
\begin{equation}\label{str}
	d'_{\Phi n}(z)=\exp\int_{0}^{z}\dfrac{\Phi(t^n)-1}{t}dt,\end{equation} which upon simplification, gives the structural formula of $d_{\Phi n}$. Similarly, we define $t_{\Phi n}(n=1,2,3,\cdots)$ by $t_{\Phi n}(0)=t'_{\Phi n}(0)-1=0$ and \[\dfrac{z t'_{\Phi n}(z)}{t_{\Phi n}(z)}=\Phi(z^n),\] which belongs to the class $\mathcal{S}^*(\Phi)$ and we write $t_{\Phi 1}$ as $t_{\Phi}$. The structural formula  for $t_{\Phi n}$ is given by:
\begin{equation}\label{tphi}
	t_{\Phi n}(z)=z\exp\int_{0}^{z}\dfrac{\Phi(t^n)-1}{t}dt.\end{equation} Note that $zd'_{\Phi n}(z)=t_{\Phi n}(z)$.
Ma-Minda~\cite{p3} proved the distortion and growth theorems for the classes $\mathcal{C}(\phi)$ and $\mathcal{S}^*(\phi)$ when $\phi\in \mathcal{M}$. Here, we prove that the result does not remain same in the case of functions in $\mathscr{M}^\circ$. It is examined with an example and which is further generalized. For this, let us consider the class $\mathcal{C}(\psi)$, the structural formula, given in~$\eqref{str}$ yields:
\begin{equation*}
	d'_{\psi}(z)=\exp\sum_{k=1}^{\infty}\dfrac{(-z)^k}{k^2}.
\end{equation*}
A numerical computation shows that $d'_{\psi}(1/2) \approx 0.63864$ and $d'_{\psi}(-1/2)\approx 1.79004$. Let the function $f$ be such that:
\begin{equation*}\label{ds}
	f'(z)=d'_{\psi 2}=\exp\sum_{k=1}^{\infty}\dfrac{(-1)^k(z)^{2k}}{2k^2},
\end{equation*}
clearly, $f \in \mathcal{C}(\psi)$. A numerical computation shows that $|f'(1/2)|\approx 0.88874$. Hence
\[d'_{\psi}(r) \leq |f'(z_0)|\leq d'_{\psi}(-r), \text{ for } z_0=r=\dfrac{1}{2}.\] Thus functions in $\mathcal{C}(\psi)$ violate distortion theorem, which shows that $\phi'(0)>0$ is inevitable in obtaining the distortion theorem of \cite{p3} for functions in $\mathcal{M}$.
\begin{rem} Let $\phi \in \mathscr{M}$ and its counter part $\Phi\in\mathscr{M}^\circ$ then
	$\Phi(\mathbb{D})=\phi(\mathbb{D})$, which implies $\mathcal{C}(\Phi)=\mathcal{C}(\phi)$ and $\mathcal{S}^*(\Phi)=\mathcal{S}^*(\phi)$. Therefore to obtain distortion and growth theorems for functions in $\mathcal{C}(\Phi)$ and $\mathcal{S}^*(\Phi)$, it is sufficient to replace $\phi(z)$ by $\Phi(-z)$,   in the result \cite[Corollary 1, p. 159]{p3}.
\end{rem}
Using the the above Remark and the fact $d_{\Phi}'(z)=d_{\phi}'(-z)$, we deduce the following result:
\begin{thm}[Distortion Theorem for $\mathcal{C}(\Phi)$]\label{distthm}
	Suppose $f \in \mathcal{C}(\Phi)$ and $|z_0|=r<1$. Then
	\begin{equation*}
		d'_{\Phi}(r) \leq |f'(z_0)|\leq d'_{\Phi}(-r).
	\end{equation*}
	Equality holds for some non zero $z_0$ if and only if $f$ is a rotation of $d_{\Phi}$, given in~$\eqref{dist}$.
\end{thm}
\begin{proof}
	Since $\mathcal{C}(\Phi)=\mathcal{C}(\phi)$, where $\phi$ is in $\mathcal{M}$, from \cite[Corollary 1]{p3}, we get the following
	\begin{equation*}
		d_{\phi}'(-r) \leq |f'(z_0)| \leq d_{\phi}'(r).
	\end{equation*}
	Now,~$\eqref{dist}$ yields $d_{\Phi}'(z)=d_{\phi}'(-z)$, which establishes the desired result.
\end{proof}
\begin{cor}[Growth Theorem for $\mathcal{C}(\Phi)$]\label{grc}
	Suppose $f \in \mathcal{C}(\Phi)$ and $|z_0|=r<1$. Then
	\begin{equation*}
		d_{\Phi}(r) \leq |f(z_0)|\leq -d_{\Phi}(-r).
	\end{equation*}
	Equality holds for some non zero $z_0$ if and only if $f$ is a rotation of $d_{\Phi}$, given in~$\eqref{dist}$.
\end{cor}
From Corollary~\ref{grc}, we get the following growth theorem:
\begin{cor}[Growth Theorem for $\mathcal{S}^*(\Phi)$]\label{grs}
	Suppose $f \in \mathcal{S}^*(\Phi)$ and $|z_0| =r<1.$ Then
	\begin{equation*}
		t_{\Phi}(r) \leq |f(z_0)|\leq -t_{\Phi}(-r).
	\end{equation*}
	Equality holds for some non zero $z_0$ if and only if $f$ is a rotation of $t_{\Phi}$, given in~$\eqref{str}$.
\end{cor}
In order to prove distortion theorem for $\mathcal{S}^*(\Phi)$, we additionally assume $\min\limits_{|z|=r}= |\Phi(z)|= \Phi(r)$ and $\max\limits_{|z|=r} |\Phi(z)|= \Phi(-r)$.
\begin{thm}[Distortion Theorem for $\mathcal{S}^*(\Phi)$]\label{dist2}
	Suppose $f \in \mathcal{S}^*(\Phi)$ and $|z_0|=r<1$. Then
	\begin{equation*}
		t'_{\Phi}(r)\leq |f'(z_0)|\leq t'_{\Phi}(-r).
	\end{equation*}
	Equality holds for some non zero $z_0$ if and only if $f$ is a rotation of $t_{\Phi}$, given in~$\eqref{str}$.
\end{thm}
\indent We now introduce the following classes involving the special type of Ma-Minda function $\psi$:
\begin{equation*}
	\mathcal{S}^*_{l}:= \left\{f \in \mathcal{S}:\dfrac{zf'(z)}{f(z)} \prec 1-\log(1+z)\right\} \text{ and }
	\mathcal{C}_{l}:= \left\{f \in \mathcal{S}:1+\dfrac{zf''(z)}{f'(z)}\prec 1-\log(1+z)\right\}.\end{equation*}
By the structural formula~$\eqref{str}$, we get a function $f \in \mathcal{S}_{l}^*$ if and only if there exists an analytic function $q$, satisfying $q(z) \prec \psi(z)$ such that
\begin{equation}\label{nstr}
	f(z)= z \exp\left(\int_{0}^{z}\dfrac{q(t)-1}{t}dt\right).
\end{equation}
Now, we give some examples of the functions in the class $\mathcal{S}_{l}^*$. For this, let us assume
\begin{equation*}\psi_1(z)=1-\dfrac{z}{6}\text{, }\psi_2(z)=\dfrac{4-z}{4+z}\text{, }\psi_3(z)=1-z \sin\dfrac{z}{4}\text{ and }\psi_4(z)=\dfrac{8-2z}{8-z}.\end{equation*}
A geometrical observation leads to $\psi_i(z)\subset \psi(z)$ $(i=1,2,3,4)$. Thus $\psi_i(z)\prec \psi(z)$. Now, the functions $f_i's$ belonging to the class $\mathcal{S}_{l}^*$ corresponding to each of the functions $\psi_i's$ are determined by the structural formula~$\eqref{nstr}$ as follows:
\[f_1(z)= z\exp\left(\dfrac{-z}{6}\right)\text{, }f_2(z)=\dfrac{16 z}{(4+z)^2}\text{, }f_3(z)=z\exp\left(4\left(1+\cos\dfrac{z}{4}\right)\right)\text{ and } f_4(z)=z-\dfrac{z^2}{8}.\] In particular, for $q(z)=\psi(z)=1-\log(1+z)$, the corresponding function obtained as follows:
\begin{equation}\label{strls}
	f_0(z)=  z \exp\left(\int_{0}^{z}\dfrac{-\log(1+t)}{t}dt\right)= z-z^2+\dfrac{3}{4}z^3-\dfrac{19}{36}z^4+\dfrac{107 }{288}z^5+\cdots,
\end{equation}
acts as an extremal function in many cases for  $\mathcal{S}_{l}^*$.
\begin{rem}\label{distrem}
	The distortion and growth theorems for $\mathcal{C}_{l}$ and $\mathcal{S}^*_{l}$ can be  obtained from that of $\mathcal{C}(\Phi)$ and $\mathcal{S}^*(\Phi)$, given in Theorem~\ref{dist}.	
\end{rem}
Here, we establish inclusion results, radius problems, majorization result and estimation of the Bloch function norm for the functions in the class $\mathcal{S}^*_{l}$. In the coefficient bound section, we consider the class:
\begin{equation}\label{mgh}
	\mathcal{A}(g,h,\phi) =:\mathcal{M}_{g,h}(\phi)= \left\{f \in \mathcal{A}: \dfrac{(f*g)(z)}{(f*h)(z)}\prec \phi(z)\text{, }\phi \in \mathscr{M}\right\},
\end{equation}
where Taylor series expansion of $g,h$ is given by~\eqref{conv} and $g_n$, $h_n>0$ with $g_n-h_n >0$. This class is defined in~\cite{murg} and authors have obtained Fekete-Szeg\"o bound for the same. We determine the bounds of fourth coefficient $|a_4|$, second Hankel determinant $|a_2a_4-a_3^2|$ and the quantity $|a_2 a_3-a_4|$ for the functions in the class $\mathcal{M}_{g,h}(\phi)$. The importance of this class lies in unification of various subclasses of $\mathcal{S}$, discussed in detail in the coefficient section. Some of our results reduce to many earlier known results of Lee et~al.~\cite{second}, Mishra et~al.~\cite{prajapat} and Singh~\cite{singh}. In view of~$\eqref{mgh}$, we also consider the class $\mathcal{M}_{g,h}(\Phi)$ for $\Phi$ in $\mathscr{M}^\circ$.
Now, we introduce the class:
$$\mathcal{M}_\alpha(\Phi)=\left\{f \in \mathcal{A}: \dfrac{zf'(z)+\alpha z^2f''(z)}{\alpha zf'(z)+(1-\alpha)f(z)}\prec \Phi(z), \;(0\leq  \alpha \leq 1)\right\}.$$
Note that when $g(z)= (z(1+(2 \alpha-1)z))/(1-z)^3$  and $h(z)=(z(1+(\alpha-1)z))/(1-z)^2$, we have $   \mathcal{M}_{g,h}(\Phi)=:\mathcal{M}_{\alpha}(\Phi)$. Further, the power series expansion of $g$ and $h$, respectively yield
\begin{equation}\label{gi}
	g_2 =~2(1+\alpha)\text{, } g_3=3(1+2\alpha)\text{, } g_4=4(1+3\alpha)\ldots \text{ and } h_2=1+\alpha\text{, }h_3=1+2\alpha\text{, } h_4=1+3\alpha\ldots.
\end{equation}
By setting $\mathcal{M}_{\alpha}(\psi)=:\mathcal{S}_{l}(\alpha)$, then  $\mathcal{S}_{l}(0)=\mathcal{S}^*_{l}$ and $\mathcal{S}_{l}(1)=\mathcal{C}_{l}$. We obtain the sharp bounds of initial coefficients such as $a_2$, $a_3$, $a_4$ and $a_5$, Fekete-Szeg\"o functional, second Hankel determinant for functions in $\mathcal{S}_{l}(\alpha)$. Further, using these sharp bounds, we estimate the third Hankel determinant bound for the functions in $\mathcal{S}_{l}(\alpha)$.
We need the following lemmas to support our results.
\begin{lem}\emph{\cite{p3}}\label{p1p2}
	Let $p \in \mathcal{P}$ be of the form $1+\sum\limits_{n=1}^{\infty}p_nz^n.$ Then
	\begin{equation*}
		|p_2 - v p_1^2| \leq
		\begin{cases}
			-4v+2, & v\leq0;\\
			2, & 0 \leq v \leq 1;\\
			4v-2, & v\geq 1.
		\end{cases}
	\end{equation*}
	When $v<0$ or $v>1$, the equality holds if and only if $p(z)$ is $(1+z)/(1-z)$ or one of its rotations. If $0<v<1$, then the equality holds if and only if $p(z)=(1+z^2)/(1-z^2)$ or one of its rotations. If $v=0$, the equality holds if and only if $p(z)=(1+\eta)(1+z)/(2(1-z))+(1-\eta)(1-z)/(2(1+z)) (0 \leq \eta \leq1)$ or one of its rotations. If $v=1$, the equality holds if and only if $p$ is the reciprocal of one of the functions such that the equality holds in the case of $v=0$. Though the above upper bound is sharp for $0<v<1$, still it can be improved as follows:
	\begin{equation}\label{C}
		|p_2-vp_1^2|+v|p_1|^2\leq 2 \quad (0<v\leq1/2) \text{ and }
		|p_2-vp_1^2|+(1-v)|p_1|^2\leq 2 \quad (1/2\leq v <1).\end{equation}
\end{lem}
\begin{lem}\emph{\cite{p1}}\label{p1p2p3}
	Let $p \in \mathcal{P}$ be of the form $1+\sum\limits_{n=1}^{\infty}p_nz^n.$ Then
	\begin{equation*}
		2p_2=p_1^2+x(4-p_1^2),
	\end{equation*}
	\begin{equation*}
		4p_3=p_1^3+2p_1(4-p_1^2)x-p_1(4-p_1^2)x^2+2(4-p_1^2)(1-|x|^2)y
	\end{equation*}
	for some $x$ and $y$ such that $|x|\leq1$ and $|y|\leq1$.
\end{lem}
The following result is proved in \cite{libera}:
\begin{lem}
	Let $p \in \mathcal{P}$ with coefficients $p_n$ as above, then
	\begin{equation}\label{p31}
		|p_3-2p_1p_2+p_1^3|\leq 2 \text{ and }	|p_1^4-3p_1^2p_2+p_2^2+2p_1p_3-p_4| \leq 2.
	\end{equation}
\end{lem}
	\section{Radius problems}\label{sec2}
	
	Besides majorization, this section chiefly focuses on estimating various radius constants associated with $\mathcal{S}_{l}^*$. We begin with establishing the following bounds meant for $\mathcal{S}_{l}^*:$
	\begin{thm}\label{bounds}
		Let $f \in \mathcal{S}_{l}^*$. Then we have for $|z|=r<1$,
		\begin{equation}\label{re}
			1-\log(1+r) \leq \RE \dfrac{z f'(z)}{f(z)}\leq 1-\log(1-r)\end{equation} and \begin{equation}\label{im}\left|\IM \dfrac{z f'(z)}{f(z)}\right|\leq \tan^{-1}\left(\dfrac{r}{\sqrt{1-r^2}}\right).
		\end{equation}
	\end{thm}
	\begin{proof}
		Since $f \in \mathcal{S}_{l}^*$, we have $z f'(z)/f(z) \prec 1-\log(1+z)$. Thus, by the definition of subordination, we have
		\begin{equation}\label{lo}
			\dfrac{z f'(z)}{f(z)} =1-\log|q(z)|-i \arg(q(z)),
		\end{equation}
		where $q(z)=1+\omega(z)$, $\omega$ be a Schwarz function satisfying $\omega(0)=0$ and $|\omega(z)|\leq |z|$. Let $q(z)=u+iv$, then $(u-1)^2+v^2<1$. For $|z|=r$, we have \begin{equation}\label{q}|q-1|\leq r.\end{equation}Squaring both sides of the above equation yields
		\begin{equation}\label{t}
			T: (u-1)^2+v^2 \leq r^2.
		\end{equation}
		Clearly $T$ represents the equation of the disk with center: $(1,0)$ and radius $r$, for which the point $(0,0)$ lies outside the disk $T$. From~$\eqref{q}$ and~$\eqref{t}$, we have
		\[|q(z)| \leq 1+r \text{ and } |q(z)| \geq  1-r.\]
		Since, $\log x$ is an increasing function on $[1,\infty)$, we have
		\begin{equation*}
			\log(1-r) \leq \log|q(z)| \leq \log(1+r).
		\end{equation*}
		Thus we get the desired result~$\eqref{re}$ by taking the real part in~$\eqref{lo}$. If $v=au$, representing the equation of the tangent to the boundary of disk $T$, which passes through the origin $O$, then the tangent and the boundary of the disk have a common point, hence from~$\eqref{t}$, we obtain\begin{equation*}
			(1+a^2)u^2-2u+1-r^2=0.
		\end{equation*}
		By the definition of tangent, we get
		\begin{equation*}
			1-(1+a^2)(1-r^2)=0 \text{, }\text{ which}\text{ yields } a=\pm \dfrac{r}{\sqrt{1-r^2}}.
		\end{equation*}
		Therefore we have
		\begin{equation*}
			\tan^{-1}\left(\dfrac{-r}{\sqrt{1-r^2}}\right) \leq \arg q(z) \leq \tan^{-1}\left(\dfrac{r}{\sqrt{1-r^2}}\right).
		\end{equation*}
		Thus we  obtain the desired result~$\eqref{im}$.
	\end{proof}
	\begin{thm}\label{rad}
		Let $f \in \mathcal{S}_{l}^*$. Then the following holds:
		\begin{itemize}
			\item[(i)] $f$ is starlike of order $\alpha$ in $|z| < \exp(1-\alpha)-1$ whenever $1-\log 2 \leq \alpha<1$.
			\item[(ii)] $f \in \mathcal{M}(\beta)$ in $|z|< 1-\exp(1-\beta)$ whenever $\beta >1$.
			\item[(iii)] $f$ is convex of order $\alpha$ in $|z| < \tilde{r}(\alpha)<1$ whenever $0\leq \alpha<1$, where $\tilde{r}(\alpha)$ is the smallest positive root of the equation:
			\begin{equation}\label{root2}
				(1-r)(1-\log(1+r))(1-\log(1+r)-\alpha)-r =0,
			\end{equation}
			for the given value of $\alpha$.
			\item[(iv)] $f$ is strongly starlike of order $\gamma$ in $|z|< r(\gamma)$ whenever $0< \gamma \leq \gamma_0 \approx 0.514674$, where
			\begin{equation}\label{rgamma}
				r(\gamma)=\sqrt{2 \left(1-\dfrac{1}{\sqrt{1+\tan^2\left(\tan \dfrac{\gamma \pi}{2}\right)}}\right)}.
			\end{equation}
			\item[(v)] $f$ is $k-$starlike function in $|z|< r(k)$ whenever $k>0$, where $r(k)$ is the smallest positive root of the equation \begin{equation}\label{root1} 1+r-e(1-r)^k=0,\end{equation} for the given value of $k$. In particular, for $k=1$, $f$ is parabolic starlike in $|z| < \tfrac{e-1}{e+1}$.
		\end{itemize}
	\end{thm}
	\begin{proof}
		(i) Since $f \in \mathcal{S}_{l}^*$, we obtain the following from Theorem~\ref{bounds}:
		\begin{equation*}
			\RE \dfrac{z f'(z)}{f(z)} \geq 1-\log(1+r), \quad |z|=r <1,\end{equation*} which yields the inequality $\RE (z f'(z))/f(z) > \alpha,$ whenever $1-\log 2 \leq \alpha <1,$ which holds true in the open disk of radius $\exp(1-\alpha)-1.$ For the function $f_0$, given in~$\eqref{strls}$ and $z_0=\exp(1-\alpha)-1$, we have $\RE(z_0 f_0'(z)/f_0(z))=\alpha.$ Hence this result is sharp.\\
		\noindent (ii) From Theorem~\ref{bounds}, we get
		\begin{equation*}
			\RE \dfrac{z f'(z)}{f(z)} \leq 1-\log(1-r),\quad |z|=r<1,
		\end{equation*}
		which yields the following inequality $\RE (z f'(z)/f(z)) < \beta,$ for $\beta >1,$
		which holds true in the open disk of radius $1-\exp(1-\beta)$. For the function $f_0$, given in~$\eqref{strls}$ and $z_0=\exp(1-\beta)-1$, we get \[\RE \dfrac{z_0 f_0'(z)}{f_0(z)}=\beta.\] Therefore the result is sharp.\\
		\noindent (iii) Let $f \in \mathcal{S}_{l}^*$. Now, $f \in \mathcal{C}_{l}(\alpha)$, whenever
		\begin{equation*}
			\RE \left(1+\dfrac{z f''(z)}{f'(z)}\right) > \alpha.
		\end{equation*}
		Since we have $ z f'(z)/f(z) =1-\log(1+\omega(z)),$ where $\omega$ is a Schwarz function, we obtain
		\begin{align}\label{8}
			\RE \left(1+\dfrac{z f''(z)}{f'(z)}\right)=\RE(1-\log(1+\omega(z)))-\RE\left(\dfrac{z \omega^{'}(z)}{(1+\omega(z))(1-\log(1+\omega(z)))}\right).
		\end{align}
		The function $\omega$ satisfies the following inequality, given in \cite{nehari}\begin{equation}\label{om}|\omega^{'}(z)| \leq \dfrac{1-|\omega(z)|^2}{1-|z|^2}.\end{equation} Using the above result,~$\eqref{8}$ reduces to
		\begin{equation*}
			\RE\left(1+\dfrac{z f''(z)}{f'(z)}\right) \geq \dfrac{(1-r)(1-\log(1+r))^2-r}{(1-r)(1-\log(1+r))}=:\nu(r).	
		\end{equation*}
		Let $\nu(r,\alpha):=\nu(r)-\alpha=(1-r)(1-\log(1+r))(1-\log(1+r)-\alpha)-r$.  Clearly, $\nu(0,\alpha)>0$ and $\nu(1,\alpha)<0$ for all $\alpha \in [0,1)$. Thus, there must exist $\tilde{r}(\alpha)$ such that $\nu(r,\alpha)\geq 0$ for all $r \in [0,\tilde{r}(\alpha)]$, where $\tilde{r}(\alpha)$ is the smallest positive root of the equation~$\eqref{root2}$. Hence the result.\\
		\noindent (iv) Since $f \in \mathcal{S}_{l}^*$, we have \[\left|\arg \dfrac{z f'(z)}{f(z)}\right| \leq |\arg(1-\log(1+z))|.\] Now, $f \in \mathcal{SS}^*(\gamma)$ whenever
		\begin{equation*}
			\left|\arg\left(1-\dfrac{1}{2}\log((1+x)^2+y^2)- i \arctan\dfrac{y}{1+x}\right)\right|< \gamma \dfrac{\pi}{2},
		\end{equation*}
		for $|z|= \sqrt{x^2+y^2}<1$. Consider the function $f_0$, given in~$\eqref{strls}$ and let us assume $$z_0=\dfrac{1}{\sqrt{1+A^2}}-1 + i \dfrac{A}{\sqrt{1+A^2}},$$ where $A=\tan(\tan \gamma \pi/2)$. We have $|z_0|= \sqrt{2 (1-1/(\sqrt{1+A^2})} <1$, whenever $ \gamma \leq \gamma_0$ and
		\begin{align*}
			\left|\arg(1-\log(1+z_0))\right|&=\left|\arg \dfrac{z_0 f_0'(z_0)}{f_0(z_0)}\right|\nonumber\\& = |\arg(1-i \arctan(\tan (\tan \gamma \pi/2)))|\nonumber\\& = \left|\arctan\left(- \arctan(\tan(\tan \gamma \pi/2))\right)\right|\nonumber\\& = \arctan(\arctan(\tan(\tan \gamma \pi/2)))\nonumber\\&= \gamma \pi/2.
		\end{align*}
		A geometrical observation reveals that $f \in \mathcal{SS}^*(\gamma)$ whenever $|z| < r(\gamma)$, where $r(\gamma)$ is given by~$\eqref{rgamma}$ and therefore the result is sharp.\\
		\noindent (v) Let $f \in \mathcal{S}_{l}^*$. Now, $f \in k-\mathcal{ST}$ whenever
		$$\RE(1-\log(1+\omega(z)))> k|\log(1+\omega(z))|\text{ } (z \in \mathbb{D}),$$
		for some Schwarz function $\omega$. Let $g(\omega(z)):=|\log(1+\omega(z))|=|\log|1+\omega(z)|+i \arg(1+\omega(z)|\text{, } \omega(z)= R e^{i t},$ where $R\leq |z|=r<1$ and $-\pi<t<\pi$. Now, consider \[g^2(R,t)=\left(\dfrac{1}{2}\log(1+R^2+2R \cos t)\right)^2+\left(\arctan\left(\dfrac{R \sin t}{1+R \cos t}\right)\right)^2,\]
		which upon partially differentiation with respect to $t$, yields
		\[h(R,t):=\dfrac{R \left(2  (R + \cos t) \arctan\dfrac{R \sin t}{1 + R \cos t} -
			\log(1 + R^2 + 2 R \cos t) \sin t\right)}{1 + R^2 + 2 R \cos t}.\]
		Then clearly, the function $h(R,t) \geq 0$ for $t \in [0,\pi)$ and $h(R,t)\leq 0$ for $t \in (-\pi,0]$. Thus \[\max_{-\pi<t<\pi}g(R,t)=\max\{g(R,-\pi),g(R,\pi)\},\] which yields
		\begin{equation}\label{1}|\log(1+\omega(z))| \leq |\log(1-R)| \leq |\log(1-r)|.\end{equation}
		The inequality~$\eqref{1}$ and Theorem~\ref{bounds} reveal that the result follows at once by showing \[1-\log(1+r) \geq k|\log(1-r)|,\] whenever $m(r,k):=1+r-e(1-r)^k \leq 0.$ 
		Clearly, $m(0,k)<0$ and $m(1,k)>0$ for fixed value of $k$. Thus, there must exist $r(k)$ such that $m(r,k)\leq 0$ for all $r \in [0,r(k)]$, where $r(k)$ is the smallest positive root of the equation~$\eqref{root1}$. Hence the result.
	\end{proof}
	Our next result involves the concept of majorization. For any two analytic functions $f$ and $g$, we say $f$ is majorized by $g$ in $\mathbb{D}$, denoted by $f$ $\ll$ $g,$ if there exists an analytic function $\mu(z)$ in $\mathbb{D}$, satisfying
	$$|\mu(z)| \leq 1 \text{ and } f(z)=\mu(z)g(z).$$
	The following majorization result involves the class $\mathcal{S}_{l}^*:$
	\begin{thm}\label{maj}
		Let $f \in \mathcal{A}$. Suppose that $f$ $\ll$ $g$ in $\mathbb{D}$, where $g \in \mathcal{S}_{l}^*$. Then
		\begin{equation*}
			|f'(z)| \leq |g'(z)|,\text{ for } |z|\leq \tilde{r},
		\end{equation*}
		where $\tilde{r}$ is the smallest positive root of the following equation:
		\begin{equation}\label{root}
			(1-r^2)(1-\log(1+r))-2r=0.
		\end{equation}
	\end{thm}
	\begin{proof}Since $g \in \mathcal{S}_{l}^*$, we have $\tfrac{z g'(z)}{g(z)} \prec 1-\log(1+z)$. Then there exists a Schwarz function $\omega(z)$ such that
		\begin{equation}\label{g}
			\dfrac{z g'(z)}{g(z)}=1-\log(1+\omega(z)).
		\end{equation}
		Let $\omega(z)= R e^{it}$, $R\leq|z|=r<1$ and $-\pi<t<\pi$. Consider the function 
		\begin{align*}
			h(R,t):=	|1-\log(1+R e^{it})|^2=\left(1-\dfrac{1}{2}\log(1+R^2+2R\cos t)\right)^2+\left(\arctan\dfrac{R \sin t}{1+R \cos t}\right)^2,
		\end{align*}
		which upon differentiation with respect to $t$ yields
		\begin{align*}h_t(R,t):=\dfrac{R \left(2 (R + \cos t) \arctan\dfrac{R \sin t}{1 + R\cos t}- (-2 +
				\log(1 + R^2 + 2 R \cos t) \sin  t\right)}{1 + R^2 + 2 R \cos t}.
		\end{align*}
		Then clearly, the function $h_t(R,t)\geq 0$ for $t \in [0,\pi)$ and  $h_t(R,t)\leq 0$ for $t \in (-\pi,0]
		$. Thus \[\min_{-\pi<t<\pi}h(R,t)=h(R,0),\] which yields
		\begin{equation*}
			|1-\log(1+\omega(z))| \geq 1-\log(1+R) \geq 1-\log(1+r).
		\end{equation*}
		Further, condition~$\eqref{g}$ yields
		\begin{equation}\label{log}
			\left|\dfrac{g(z)}{g'(z)}\right| =\dfrac{|z|}{|1-\log(1+\omega(z)|}\leq \dfrac{r}{1-\log(1+r)}.
		\end{equation}
		By the definition of majorization, we get $f(z)=\mu(z) g(z)$, which upon differentiation, gives
		\begin{equation}\label{f}
			f'(z)=\mu(z)g'(z)+g(z)\mu'(z)=g'(z)\left(\mu(z)+\mu'(z)\dfrac{g(z)}{g'(z)}\right).
		\end{equation}
		The function $\mu$ satisfies the inequality~$\eqref{8}$, thus using~$\eqref{8}$ for $\mu$ and substituting the inequality~$\eqref{log}$ in~$\eqref{f}$, we get
		\begin{equation*}
			|f'(z)| \leq K(r, \zeta)|g'(z)|,
		\end{equation*}
		where $K(r,\zeta)= \zeta+\tfrac{r (1-\zeta^2)}{(1-r^2)(1-\log(1+r))}$ for $|\mu(z)|=\zeta$ $(0\leq \zeta \leq 1)$. Let us choose \begin{equation*} \eta(r,\zeta):=1-K(r,\zeta)=(1-r^2)(1-\log(1+r))-r(1+\zeta).\end{equation*} For $\zeta=1$, $\eta(r,\zeta)$ attains its minimum value, which is given by \[\eta(r,1)=:\nu(r)=(1-r^2)(1-\log(1+r))-2r.\] Clearly, $\nu(0)=1>0$ and $\nu(1)=-2<0$. In view of these inequalities there must exist $\tilde{r}$ such that $\nu(r) \geq 0$ for all $r \in [0,\tilde{r}]$, where $\tilde{r}$ is the smallest positive root of the equation~\eqref{root}. Hence the proof is complete.
	\end{proof}
	\section{Inclusion relations}\label{sec3}
	In this section, we give inclusion relations between the classes $\mathcal{S}^*_{l}$ and various other subclasses of starlike functions, namely $\mathcal{S}^*(\alpha)$, $\mathcal{SS}^*(\gamma)$, $\mathcal{ST}(1,\alpha)$ and $\mathcal{S}^*(q_c)$.
	\begin{thm}
		The class $\mathcal{S}^*_{l}$ satisfies the following results:
		\begin{itemize}
			\item[(i)] $\mathcal{S}^*_{l} \subset \mathcal{S}^*(\alpha) \subset \mathcal{S}^*$ for $0 \leq \alpha \leq 1-\log 2$.
			\item[(ii)]$\mathcal{S}^*_{l} \subset \mathcal{SS}^*(\gamma) \subset \mathcal{S}^*$ for $2 \tilde{f}(\theta_0)/\pi \leq \gamma \leq 1$, where $\theta_0$ is the smallest positive root of the equation $-2 + \log(2 (1 + \cos \theta)) + \theta \tan \theta/2=0$ and $\tilde{f}(\theta) = \arg(1-\log(1+e^{i \theta}))$, $\theta \in [0,\pi)$.
			\item[(iii)]$\mathcal{S}^*_{l} \subset \mathcal{ST}(1,\alpha)$ for $\alpha \leq 1-2\log 2$.
			\item[(iv)]$\mathcal{S}^*(q_c) \subset \mathcal{S}^*_{l} \subset \mathcal{S}^*$ for $c \leq c_0$, where $c_0 = \log 2(2-\log 2)$.	
		\end{itemize}
		The above constants in each part is best possible. The pictorial representation of the result is depicted in the \textbf{Figure 1}.
	\end{thm}
	\begin{figure}[ht]
		\begin{minipage}{0.55\textwidth}
			\includegraphics[width=0.95\linewidth, height=8cm]{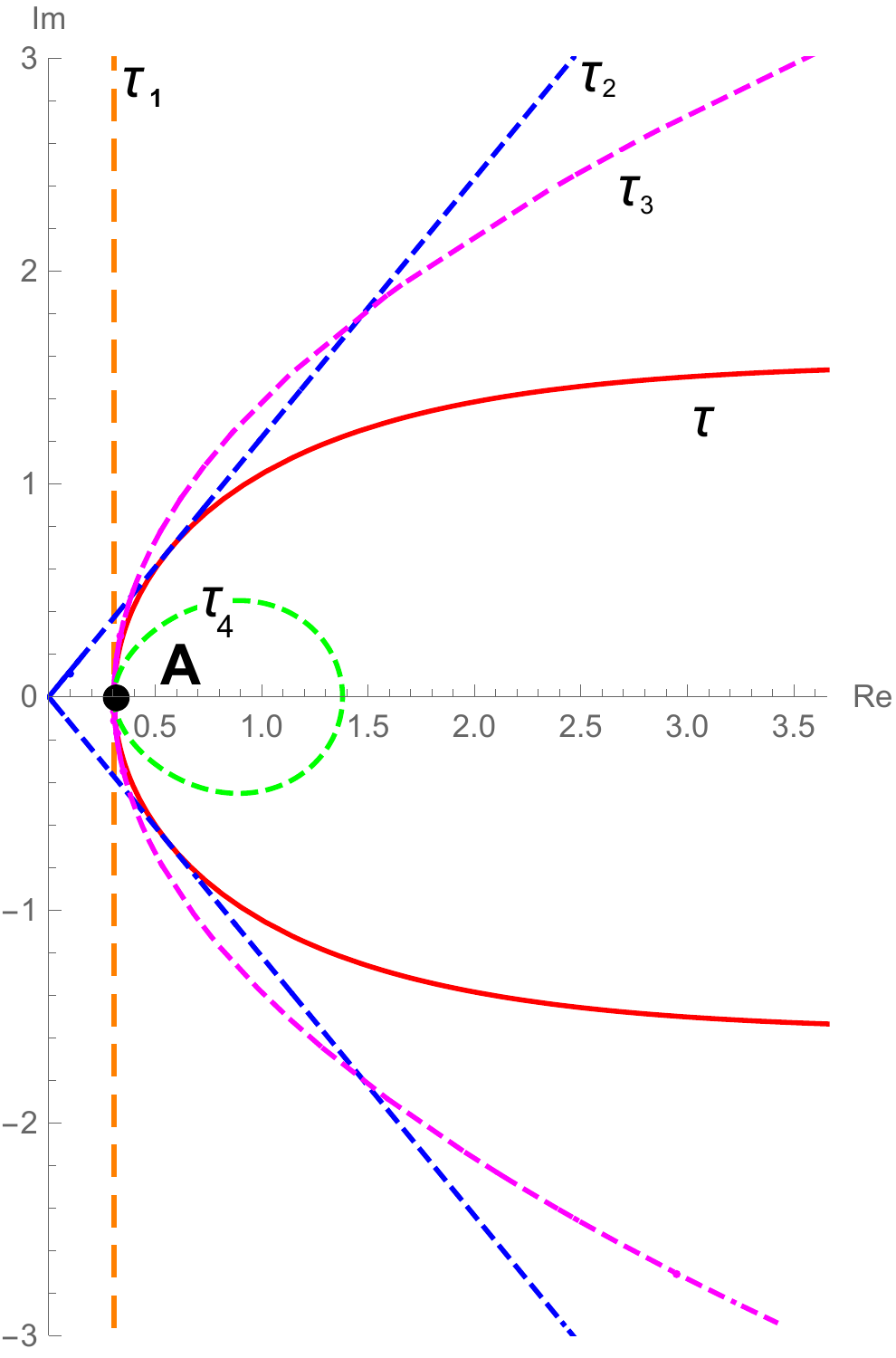}
		\end{minipage}%
		\begin{minipage}{0.55\textwidth}
			\parbox{0.9\linewidth}{\large
				$\tau:|\exp(1-w)-1|=1$ \\ \ \\
				$\tau_1:\RE w=1-\log 2$\\ \ \\$\tau_2:|\arg w|=\dfrac{7029}{12500}\dfrac{\pi}{2}$\\ \ \\$\tau_3: \RE w-|w-1|=1-2\log 2$\\ \ \\$\tau_4:|w^2(z)-1|=\log 2(2-\log 2)$\\ \ \\ $A=1-\log 2$
			}
		\end{minipage}
		\caption{Boundary curves of best dominants and subordinant of $\psi(z)=1-\log(1+z)$.}
		\label{fig:incl_rel}
	\end{figure}
	\begin{proof}
		(i) Since $f \in \mathcal{S}^*_{l}$, we have $z f'(z)/f(z) \prec 1-\log(1+z)$. Theorem~\ref{bounds} yields the following:
		\begin{equation*}
			1-\log 2=\min_{|z|=1} \RE (1-\log(1+z)) < \RE \dfrac{z f'(z)}{f(z)}.
		\end{equation*}
		Hence, the result follows.\\
		\noindent (ii) Let $f \in \mathcal{S}^*_{l}$. Thus, we have
		\begin{align*}
			\left|\arg \dfrac{z f'(z)}{f(z)}\right| &< \max_{|z|=1} |\arg (1-\log(1+z))|\\&= \max_{-\pi\leq\theta\leq\pi} \left|\arctan\left(\dfrac{-\theta}{2-\log(2(1+\cos \theta))}\right)\right|\\&=: \max_{-\pi\leq\theta\leq\pi} |\tilde{f}(e^{i \theta})|.
		\end{align*}
		Due to the symmetricity of the function $\tilde{f}(\theta)$, we consider $\theta\in [0,\pi]$ and
		$\tilde{f}'(\theta)=0$ yields $-2 + \log(2 (1 + \cos(\theta_0))) + \theta_0 \tan(\theta_0/2)=0$, where $\theta_0 \approx 1.37502$. A calculation shows that $\tilde{f}''(\theta) <0$, which implies $\max_{0\leq\theta\leq\pi}\tilde{f}( \theta)=\tilde{f}(\theta_0)\approx 0.88329$. Thus $f\in \mathcal{SS}^*(\gamma)$, for the given range of $\gamma$.\\
		\noindent (iii) Let us consider the domain $\Omega_{\alpha}:= \{w \in \mathbb{C}: \RE w > |w-1|+\alpha\}$, whose boundary represents a parabola, for $w=x+i y$, given by:
		\begin{equation*}
			x= \dfrac{y^2}{2(1-\alpha)}+\dfrac{1+\alpha}{2},
		\end{equation*}
		whose vertex is given by: $((1+\alpha)/2,0)$. In order to prove the result, it suffices to show
		\begin{align*}
			h(\theta)&:= \RE (1-\log(1+z)) -|\log(1+z)| \\&~=1-\dfrac{1}{2}\log(2(1+\cos \theta)) -\sqrt{\dfrac{1}{4}\log^2(2(1+\cos \theta))+\dfrac{\theta^2}{4}}> \alpha,
		\end{align*}
		for $z=e^{i \theta}$. A numerical computation shows that $\min_{-\pi \leq \theta \leq \pi} h(\theta) = h(0)=1-2 \log 2$. Hence, the result.\\
		\noindent (iv) Since $f \in \mathcal{S}^*(q_c)$, we have $z f'(z)/f(z) \prec \sqrt{1+cz}$ and
		\[\sqrt{1-c}= \min_{|z|=1} \sqrt{1+cz} < \RE \dfrac{z f'(z)}{f(z)} < \max_{|z|=1} \sqrt{1+cz} =\sqrt{1+c}.\] Similar analysis can be carried out for the imaginary part bounds and therefore by using Theorem~\ref{bounds}, we get the result.
	\end{proof}
\section{Coefficient bounds}
This section, deals with various coefficient related bound estimates. Here we need the function $H(q_1,q_2)$, given in \cite[Lemma~3]{pokhrov}, to establish our results in what follows. Evidently the class $\mathcal{M}_{g,h}(\phi)$ unifies various subclasses of $\mathcal{S}$ for different choices of $g$ and $h$. A few of the same are enlisted below for ready reference:
\begin{align}\label{convo}
	\dfrac{(f*g)(z)}{(f*h)(z)} = \begin{cases}
		\dfrac{z f'(z)}{f(z)}, & g(z)=\dfrac{z}{(1-z)^2}\text{, }h(z)= \dfrac{z}{1-z};\\ \ \\ \dfrac{zf'(z)+\alpha z^2f''(z)}{f(z)}, & g(z)=\dfrac{z(1+(2\alpha-1)z)}{(1-z)^3}\text{, } h(z)=\dfrac{z}{1-z};\\ \ \\ \dfrac{(zf'(z)+\alpha z^2f''(z))'}{f'(z)}, & g(z)=\dfrac{z(1-z^2 +2\alpha z (2 + z))}{(1 -z)^4}\text{, } h(z)=\dfrac{z}{(1-z)^2};\\ \dfrac{2 zf'(z)}{f(z)-f(-z)}, & g(z)=\dfrac{z}{(1-z)^2}\text{, } h(z)=\dfrac{z}{(1-z^2)};\\ \ \\ \dfrac{(2zf'(z))'}{(f(z)-f(-z))'}, &g(z)=\dfrac{z(1+z)}{(1-z)^3}\text{, } h(z)=\dfrac{z(1+z^2)}{(1-z^2)^2};\\ \ \\ \dfrac{zf'(z)+\alpha z^2f''(z)}{\alpha zf'(z)+(1-\alpha)f(z)}, &g(z)=\dfrac{z(1+(2\alpha-1)z)}{(1-z)^3} \text{, } h(z)=\dfrac{z(1+(\alpha-1)z)}{(1-z)^2}. \end{cases} \end{align}
\begin{rem}
	Murugusundaramoorthy~et~al. \cite[Theorem~6.1]{murg} obtained the result of Fekete-Szeg\"o functional bound for  functions in the class $\mathcal{M}_{g,h}(\phi)$. Since $\mathcal{M}_{g,h}(\Phi(z))$ $=\mathcal{M}_{g,h}(\phi(z))$, we state below the parallel result for functions in the class $\mathcal{M}_{g,h}(\Phi(z))$ by simply replacing each $B_i$ by $(-1)^iC_i$, the result is needed to prove our subsequent example.
\end{rem}
\begin{thm}\label{fekete}
	For $f \in \mathcal{M}_{g,h}(\Phi)$, we have
	\begin{align*}
		|a_3 - t a_2^2| \leq \begin{cases}
			\dfrac{C_2}{g_3-h_3}-\dfrac{t C_1^2}{(g_2-h_2)^2}+\dfrac{(g_2 h_2-h_2^2)C_1^2}{(g_3-h_3)(g_2-h_2)^2}, & t\leq \kappa_1;\\
			\dfrac{-C_1}{g_3-h_3}, & \kappa_1 \leq t \leq \kappa_2;\\
			\dfrac{-C_2}{g_3-h_3}+\dfrac{t C_1^2}{(g_2-h_2)^2}-\dfrac{(g_2 h_2-h_2^2) C_1^2}{(g_3-h_3)(g_2-h_2)^2}, & t\geq \kappa_2,
		\end{cases}
	\end{align*}
	where
	\begin{equation*}
		\kappa_1 = \dfrac{(g_2-h_2)^2(C_2+C_1)+h_2(g_2-h_2)C_1^2}{(g_3-h_3)C_1^2} \text{ and } \kappa_2 = \dfrac{(g_2-h_2)^2(C_2-C_1)+h_2(g_2-h_2)C_1^2}{(g_3-h_3)C_1^2}.
	\end{equation*}
\end{thm}
The result is sharp whenever $f$ satisfies:
\begin{align*}
	\dfrac{(f*g)(z)}{(f*h)(z)} =\begin{cases}
		\Phi(z), & t< \kappa_1\text{ or } t>\kappa_2;\\
		\Phi(z^2), & \kappa_1 < t < \kappa_2;\\
		\Phi(\Psi(z)), & t=\kappa_1;\\
		\Phi(-\Psi(z)), & t=\kappa_2,
	\end{cases}
\end{align*}
where $\Psi(z)=\dfrac{z(z+\eta)}{1+\eta z}\quad (0\leq \eta \leq 1)$.
\begin{rem}
	We notice that the bound of Fekete-Szeg\"o stated in \cite[Theorem 6.1]{murg}, namely $|a_3 -\mu a_2^2| \leq B_1/2(g_3-h_3) \text{ when } \sigma_1\leq \mu \leq \sigma_2$, is incorrect and should be $|a_3 -\mu a_2^2| \leq B_1/(g_3-h_3)$, which is appropriately corrected in  Theorem~\ref{fekete}.
\end{rem}
In the following example, we establish a Fekete-Szeg\"o result for the class $\mathcal{S}_{l}(\alpha)$:
\begin{example}\label{feke}
	Let $f \in \mathcal{S}_{l}(\alpha)$. Then
	\begin{equation*}
		|a_3-t a_2^2|\leq \begin{cases}\dfrac{3}{4(1+2\alpha)}-\dfrac{t}{(1+\alpha)^2}, & t \leq \dfrac{(1+\alpha)^2}{4(1+2\alpha)}=:\kappa_1;\\
			\dfrac{1}{2(1+2\alpha)}, &\dfrac{(1+\alpha)^2}{4(1+2\alpha)}\leq t\leq \dfrac{5(1+\alpha)^2}{4(1+2\alpha)};\\
			\dfrac{t}{(1+\alpha)^2}-\dfrac{3}{4(1+2\alpha)}, &t \geq\dfrac{5(1+\alpha)^2}{4(1+2\alpha)}=:\kappa_2.\end{cases}
	\end{equation*}
	The result is sharp.
\end{example}
\begin{proof}
	Since $f \in\mathcal{S}_{l}(\alpha)= \mathcal{M}_{\alpha}(\psi(z))$, we have $ C_1=-1$, $C_2=1/2$ and $C_3=-1/3$. The result follows from Theorem~\ref{fekete} by substituting the values of $g_i's$ and $h_i's$ from~$\eqref{gi}$. Equality holds whenever $f$ satisfies:
	\begin{align*}
		\dfrac{(f*g)(z)}{(f*h)(z)} =\begin{cases}
			1-\log(1+z), & t< \kappa_1\text{ or } t>\kappa_2;\\
			1-\log(1+z^2), & \kappa_1 < t < \kappa_2;\\
			1-\log(1+\tfrac{z(z+\eta)}{1+\eta z})), & t=\kappa_1;\\
			1-\log(1-\tfrac{z(z+\eta)}{1+\eta z}), & t=\kappa_2.
		\end{cases}
	\end{align*}
\end{proof}
\begin{example}\label{exfek}
	Let $f \in \mathcal{S}_{l}(\alpha)$. Then
	\begin{align*}\label{fek}
		\text{(i)} \quad |a_3-a_2^2|\leq \dfrac{1}{2(1+2\alpha)},\quad \text{(ii)}\quad  |a_3| \leq \dfrac{3}{4(1+2\alpha)}.
	\end{align*}
	These results are sharp.
\end{example}
The proof directly follows from Example~\ref{feke}.
\begin{thm}\label{sec}
	Let $f \in \mathcal{M}_{g,h}(\phi)$ and either \begin{equation}\label{l1} (g_3-h_3)^2\leq L \quad \text{ or }\quad L <	(g_3-h_3)^2 \leq 2L, \end{equation}
	where $L=(g_2-h_2)(g_4-h_4)$, then
	\begin{itemize}
		\item[\emph{(1)}]	\begin{equation*}
			\text{ }|a_2 a_4-a_3^2| \leq \dfrac{B_1^2}{(g_3-h_3)^2},
		\end{equation*}
		whenever  $B_1$, $M$ and $T$ satisfy the conditions
		\begin{equation}\label{m}
			|M|-B_1^2(g_2-h_2)^4(g_4-h_4) \leq 0 \text{ and } |T|+ B_1(g_3-h_3)^2(g_2-h_2)-2B_1(g_2-h_2)^2(g_4-h_4) \leq 0.\end{equation}
		\item[\emph{(2)}] 	\begin{equation*}
			|a_2 a_4-a_3^2| \leq \dfrac{|M|}{(g_2-h_2)^4(g_3-h_3)^2(g_4-h_4)},
		\end{equation*}
		whenever  $B_1$, $M$ and $T$ satisfy the conditions
		\begin{equation*}
			| T|+ B_1(g_3-h_3)^2(g_2-h_2)-2B_1(g_2-h_2)^2(g_4-h_4)\geq 0
		\end{equation*}and \begin{equation*}2|M|-B_1|T|(g_2-h_2)^2-B_1^2(g_3-h_3)^2(g_2-h_2)^3 \geq 0\end{equation*}or \begin{equation*}
			|T|+ B_1(g_3-h_3)^2(g_2-h_2)-2B_1(g_2-h_2)^2(g_4-h_4)\leq 0\end{equation*}and \begin{equation*}|M|-B_1^2(g_2-h_2)^4(g_4-h_4) \geq 0.  \end{equation*}
		\item[\emph{(3)}] 	\begin{align*}
			|a_2 a_4-a_3^2|\leq &  -\dfrac{1}{(|M|-B_1|T|(g_2-h_2)^2-B_1^2(g_3-h_3)^2(g_2-h_2)^3+B_1^2(g_2-h_2)^4(g_4-h_4))}\times\\&\dfrac{B_1^2(|T|+B_1(g_3-h_3)^2(g_2-h_2)-2 B_1 (g_2-h_2)^2(g_4-h_4))^2}{4(g_3-h_3)^2(g_4-h_4)}+\dfrac{B_1^2}{(g_3-h_3)^2},
		\end{align*}
		whenever  $B_1$, $M$ and $T$ satisfy the conditions
		\begin{equation}\label{3}
			|T|+ B_1(g_3-h_3)^2(g_2-h_2)-2B_1(g_2-h_2)^2(g_4-h_4)> 0
		\end{equation}
		and\begin{equation}\label{mt}
			2|M|-B_1|T|(g_2-h_2)^2-B_1^2(g_3-h_3)^2(g_2-h_2)^3 \leq 0,
	\end{equation}\end{itemize}
	where \begin{align}\label{M} M&=~B_1^4\Bigg(-h_2^2(g_2-h_2)^2(g_4-h_4)+(g_3-h_3)\bigg(g_2g_3h_2^2-g_3h_2^3+g_2^2h_2h_3-3g_2h_2^2h_3+2h_2^3h_3\nonumber\\&\quad +(g_3-h_3)(-g_2h_2^2+h_2^3)\bigg)\Bigg)-B_2^2(g_2-h_2)^4(g_4-h_4)+B_1B_3(g_3-h_3)^2(g_2-h_2)^3\nonumber \\&\quad+B_1^2B_2\bigg((g_3-h_3)(g_2-h_2)^2\Big(g_3h_2+g_2h_3-2h_2h_3-2h_2(g_2-h_2)(g_4-h_4)\Big)\bigg)\end{align} and
	\begin{align}\label{T}
		T=~&2B_2(g_2-h_2)^2(g_4-h_4)+2B_1^2h_2(g_2-h_2)(g_4-h_4)-B_1^2g_3h_2(g_3-h_3)\nonumber \\&-B_1^2g_2h_3(g_3-h_3)+2B_1^2h_2h_3(g_3-h_3)-2B_2(g_3-h_3)^2(g_2-h_2).
	\end{align}
\end{thm}
\begin{proof}
	The series expansion of the functions $f$, $g$ and $h$ yields
	\begin{equation}\label{gh}
		\dfrac{(f*g)(z)}{(f*h)(z)}=1+ a_2(g_2-h_2)z+(a_3(g_3-h_3)+a_2^2h_2(h_2-g_2))z^2+\cdots.
	\end{equation}
	Here, we define a function $p$ in $\mathcal{P}$ as follows:
	\begin{equation}\label{omega}
		p(z)= \dfrac{1+\omega(z)}{1-\omega(z)} = 1+p_1z+p_z z^2+\cdots.	
	\end{equation}
	Then, we have $\omega(z)=\tfrac{p(z)-1}{p(z)+1}.$
	Clearly, $\omega$ is a Schwarz function. Since $(f*g)(z)/(f*h)(z) \prec \phi(z)$, we get \begin{equation}\label{main} \dfrac{(f*g)(z)}{(f*h)(z)} = \phi(\omega(z)).\end{equation} Now, using~$\eqref{gh}$,~$\eqref{phi}$ and expression of $\omega$ in terms of $p$ in $\eqref{main}$, we get
	\begin{equation}\label{a2a3}
		a_2 = \dfrac{B_1 p_1}{2(g_2-h_2)}, \text{ } a_3 =\dfrac{B_2 p_1^2(g_2-h_2)-B_1(p_1^2-2p_2)(g_2-h_2)+B_1^2p_1^2h_2}{4 (g_2-h_2)(g_3 - h_3)}
	\end{equation}
	and
	\begin{align}\label{a4}
		a_4=~&\dfrac{1}{8 (g_2 - h_2) (g_3 - h_3) (g_4 -
			h_4)}\Bigg(p_1 (-2 B_2 p_1^2 + B_3 p_1^2 + 4 B_2 p_2) (g_2 - h_2)(g_3 - h_3)\nonumber \\&+B_1^3 p_1^3 h_2 h_3 - B_1^2 p_1 (p_1^2 - 2 p_2) \Big(g_3 h_2 + (g_2 - 2 h_2) h_3\Big) + B_1 \bigg(p_1^3 \Big(g_2 (g_3 + (B_2-1) h_3)\nonumber \\& +
		h_2 \big((B_2-1) g_3 + h_3 - 2 B_2 h_3\big)\Big)-4 p_1 p_2 (g_2 - h_2) (g_3 - h_3) + 4 p_3 (g_2 - h_2)(g_3 - h_3)\bigg)\Bigg).
	\end{align}
	We assume  $p_1=:p \in [0,2]$ and upon substituting the values of $p_2$ and $p_3$, given in  Lemma~\ref{p1p2p3}, in the expression $a_2 a_4-a_3^2$, we get
	\begin{align*}
		a_2a_4-a_3^2:=&\dfrac{1}{16(g_2-h_2)^4(g_3-h_3)^2(g_4-h_4)}\bigg(p^4 M-p^2x(4-p^2)(g_2-h_2)^2B_1T-(4-p^2)^2x^2\nonumber \\&\quad B_1^2(g_2-h_2)^4(g_4-h_4)-p^2x^2(4-p^2)B_1^2(g_3-h_3)^2(g_2-h_2)^3\nonumber \\&+2B_1^2p(4-p^2)(g_3-h_3)^2(g_2-h_2)^3y(1-|x|^2)\bigg),
	\end{align*}
	where $M$ and $T$ are given in~$\eqref{M}$ and~$\eqref{T}$, respectively.
	Applying triangular inequality in the above equation with the assumption that $|x|=\rho$, we get
	\begin{align*}
		|a_2a_4-a_3^2|\leq& \dfrac{1}{16 (g_2-h_2)^4(g_3-h_3)^2(g_4-h_4)}\bigg(|M|p^4+p^2(4-p^2)\rho^2B_1^2(g_3-h_3)^2(g_2-h_2)^3)\nonumber\\&+B_1|T| p^2 \rho (4-p^2) ~(g_2-h_2)^2+2B_1^2(g_3-h_3)^2(g_2-h_2)^3p(4-p^2)(1-\rho^2)\nonumber\\&+B_1^2(g_2-h_2)^4(g_4-h_4)(4-p^2)^2\rho^2\bigg)= :G(p,\rho).
	\end{align*}
	The function $G(p,\rho)$ is an increasing function of $\rho$ in the closed interval $[0,1]$, when either of the conditions in $\eqref{l1}$ hold. Thus $\max_{0\leq \rho \leq 1}G(p,\rho)=G(p,1)=:F(p)$.
	On solving further, we get
	\begin{align}\label{eq}
		F(p):=&\dfrac{1}{16(g_2-h_2)^4(g_3-h_3)^2(g_4-h_4)}\Bigg(\bigg(|M|-B_1(g_2-h_2)^2\Big(B_1(g_2-h_2)^2(g_4-h_4)-|T|\nonumber\\&-B_1(g_2-h_2)(g_3-h_3)^2\Big)\bigg)p^4+4B_1(g_2-h_2)^2\Big(|T|+B_1(g_2-h_2)(g_3-h_3)^2\nonumber\\&-2B_1(g_2-h_2)^2(g_4-h_4)\Big)p^2+16B_1^2(g_2-h_2)^4(g_4-h_4)\Bigg)\nonumber\\&=:\dfrac{1}{16(g_2-h_2)^4(g_3-h_3)^2(g_4-h_4)}\bigg(Ap^4+Bp^2+C\bigg).
	\end{align}
	We recall that
	\begin{equation}\label{ABC1}
		\max_{0\leq t \leq 4}(At^2+Bt+C)=
		\begin{cases}
			C, & B\leq 0, A\leq -\tfrac{B}{4};\\
			16A+4B+C, & B\geq 0\text{, } A\geq -\tfrac{B}{8}\text{ or }B\leq 0\text{, } A\geq -\tfrac{B}{4};\\
			\dfrac{4AC-B^2}{4A},& B>0\text{, } A\leq -\tfrac{B}{8}.
		\end{cases}
	\end{equation}
	From~$\eqref{eq}$ and~$\eqref{ABC1}$, we get the desired result.
\end{proof}
\begin{rem} In view of the first case  of~$\eqref{convo}$, Theorem~\ref{sec} reduces to the result obtained by Lee~et~al.~\cite{second} which gives the sharp second Hankel determinant bound for functions in the class $\mathcal{S}^*(\phi)$. 
\end{rem}
\begin{rem}
	We notice that the bound evaluated in \cite[Theorem~1]{second} for the case $3$, reproduced below:
	\[|a_2 a_4-a_3^2|\leq \dfrac{B_1^2}{12}\left(\dfrac{3|4 B_1 B_3 -B_1^4-3 B_2^2|-4B_1 |B_2|+4B_1^2-|B_2|^2}{|4 B_1 B_3-B_1^4-3 B_2^2|-2 B_1 |B_2|-B_1^2}\right),\] has an error and it should be: \[|a_2 a_4-a_3^2|\leq \dfrac{B_1^2}{12}\left(\dfrac{3|4 B_1 B_3 -B_1^4-3 B_2^2|-4B_1 |B_2|-4B_1^2-|B_2|^2}{|4 B_1 B_3-B_1^4-3 B_2^2|-2 B_1 |B_2|-B_1^2}\right).\]
\end{rem}
\begin{rem}
	In view of the sixth case of~$\eqref{convo}$ for, $\phi(z)=\tfrac{1+z}{1-z}$, Theorem~\ref{sec} reduces to a result obtained in \cite{singh}.
\end{rem}
\begin{rem}\label{ci}
	The second Hankel determinant bound for the functions in the class $\mathcal{M}_{g,h}(\Phi)$ can be obtained from Theorem~\ref{sec} by replacing each $B_{i}$ by $(-1)^iC_{i}$.
\end{rem}
\begin{example}\label{exsec}
	Let $f \in \mathcal{S}_{l}(\alpha)$. Then
	\begin{align*}
		|a_2 a_4-a_3^2| \leq \begin{cases}
			\dfrac{1}{4(1+2\alpha)^2}, & 0\leq \alpha\leq\tfrac{2+\sqrt{15}}{11};\\
			\dfrac{31\alpha^4+136\alpha^3-14\alpha^2-24\alpha-3}{2(61\alpha^2-20\alpha-5)(1+\alpha)(1+3\alpha)(1+2\alpha)^2}, & \tfrac{2+\sqrt{15}}{11} \leq \alpha \leq 1.
		\end{cases}
	\end{align*}
\end{example}
\begin{proof}
	For $f \in \mathcal{M}_{g,h}(\psi)$, we have $C_1=-1$, $C_2=1/2$ and $C_3=-1/3$. Now, $\mathcal{M}_{\alpha}(\phi(z))=\mathcal{M}_{\alpha}(\phi(-z))=\mathcal{M}_{\alpha}(\Phi(z))$. Thus using Theorem~\ref{sec} and Remark~\ref{ci}, we get
	\begin{equation*}
		M = \dfrac{1}{12}(1+\alpha)^3 (61 \alpha^2-20 \alpha-5) \text{ and } T=-(1+5a+17 \alpha^2+13 \alpha^3).
	\end{equation*}
	To get the desired estimate, we now consider the following cases:
	\begin{itemize}
		\item[(i)] For $0 \leq \alpha \leq (2+\sqrt{15})/11$, it is easy to verify that $M$ and $T$ satisfy the inequalities given in~$\eqref{m}$, respectively.
		\item[(ii)] For $ (2+\sqrt{15})/11 < \alpha \leq 1$, it is easy to verify that the inequalities~$\eqref{3}$ and $\eqref{mt}$ hold true for $M$ and $T$. 															\end{itemize}
	Now, the assertion follows at once from Theorem~\ref{sec}.
\end{proof}
We recall that the function $f \in \mathcal{S}$ is in the class $\mathcal{S}_s^*(\phi)$ if it satisfies
\begin{equation*}
	\dfrac{2 zf'(z)}{f(z)-f(-z)} \prec \phi(z), \text{ } z \in \mathbb{D}
\end{equation*}
and is in the class $\mathcal{C}_s(\phi)$ if it satisfies
\begin{equation*}
	\dfrac{(2zf'(z))'}{(f(z)-f(-z))'} \prec \phi(z), \text{ } z \in \mathbb{D}.
\end{equation*}
The following couple of corollaries can be obtained from Theorem~\ref{sec}, in view of fourth and fifth cases  of~$\eqref{convo}$ respectively.
\begin{cor} \label{sym}
	Let $f \in \mathcal{S}_s^*(\phi)$. Then  we have
	\begin{align*}
		|a_2 a_4-a_3^2| \leq \begin{cases}
			B_1^2/4,& \text{ when~A~holds};\\ |M|/256, & \text{ when~B~holds};\\ \dfrac{B_1^2}{4}-\dfrac{B_1^2(|T|-24 B_1)^2}{64(|M|-4 B_1 |T|+32B_1^2)}, &\text{ when~C~holds},
		\end{cases}
	\end{align*}
	where $M=16B_1^2B_2-64B_2^2+32B_1B_3 \text{, } T= 16B_2-4B_1^2$ and
	\begin{align*}
		&A: 	|M|-64 B_1^2 \leq 0 \text{ and } |T|-24 B_1 \leq 0.\\ &B:|M|-2B_1 |T|-16 B_1^2 \geq 0 \text{ and } |T|-24B_1 \geq 0,
		\text{ or }\\&\quad~~
		|M|-64 B_1^2 \geq 0 \text{ and }  |T|-24 B_1 \leq 0.\\&C: 	|M|-2B_1 |T|-16B_1^2 \leq 0 \text{ and } |T|-24B_1 >0.
	\end{align*}
\end{cor}
\begin{cor}\label{csym}
	Let $f \in \mathcal{C}_s(\phi)$. Then we have
	\begin{align*}
		|a_2 a_4-a_3^2| \leq
		\begin{cases}
			B_1^2/36,& \text{when~A~holds};\\|M|/147456,& \text{ when~B~holds};\\ \dfrac{B_1^2}{36}-\dfrac{B_1^2(|T|-368 B_1)^2}{2304(|M|-16B_1|T|+1792 B_1^2)},& \text{ when~C~holds},
		\end{cases}
	\end{align*}
	where $	M=128(9 B_1^2 B_2-32 B_2^2 +18 B_1 B_3)\text{, } T= 8(28 B_2 -9 B_1^2)$ and,
	\begin{align*}
		&A:	|M|-4096B_1^2 \leq 0 \text{ and } |T|-368 B_1\leq 0.\\& B: 	|M|-8B_1|T|-1152B_1^2 \geq 0 \text{ and } |T|-368B_1\geq 0,\text{ or }\\&\quad~~|T|-368B_1 \leq 0 \text{ and } |M|-4096B_1^2 \geq 0.\\&C: 	|T|-368B_1 > 0 \text{ and } |M|-8B_1|T|-1152B_1^2 \leq 0.
	\end{align*}
\end{cor}
\begin{rem}
	When $\phi(z)=\tfrac{1+z}{1-z}$, the Corollaries~\ref{sym} and \ref{csym} reduce to  the results obtained in \cite{prajapat} for the classes $\mathcal{S}_s^*$ and $\mathcal{C}_s$ of starlike functions and convex functions with respect to symmetric points  respectively.
\end{rem}
\begin{rem}
	Note that the second Hankel determinant bound for the functions in the classes $\mathcal{S}_s^*(\Phi)$ and $\mathcal{C}_s(\Phi)$ can be obtained from the Corollaries~\ref{sym} and \ref{csym}, respectively by replacing each $B_{i}$ by $(-1)^iC_{i}$.
\end{rem}
Expressing the fourth coefficient $a_4$ for the function $f \in \mathcal{M}_{g,h}(\phi)$ in terms of the Schwarz function $\omega(z)= 1+\omega_1 z+\omega_2 z^2+\cdots$, we obtain the bound of $a_4$ as follows:
\begin{equation}\label{a4a}
	|a_4| \leq \dfrac{B_1}{g_4-h_4}H(q_1, q_2),
\end{equation}
where
\begin{equation}\label{q1}
	q_1 =\dfrac{2B_2 (g_2-h_2)(g_3-h_3)+B_1^2(g_3 h_2+g_2 h_3-2h_2 h_3)}{B_1 (g_2-h_2)(g_3-h_3)}
\end{equation}
and  \begin{equation}\label{q2}
	q_2=\dfrac{B_3 (g_2-h_2)(g_3-h_3)+B_1^3 h_2 h_3 +B_1 B_2 (g_3h_2+g_2 h_3 -2h_2 h_3)}{B_1 (g_2-h_2)(g_3-h_3)}.
\end{equation}
\begin{rem}
	In view of first case of~$\eqref{convo}$, for $\phi(z)=\sqrt{1+z}$, the above result~\eqref{a4a} reduces to the result obtained in \cite{raza}.
\end{rem}
\begin{rem}\label{4rem}
	Note that the bound for the fourth coefficient for the functions in the class $\mathcal{M}_{g,h}(\Phi)$ can be obtained from \eqref{a4a},~\eqref{q1} and ~\eqref{q2} by replacing each $B_{i}$ by $(-1)^iC_{i}$.
\end{rem}
\begin{example}\label{exfourth}
	Let $f \in \mathcal{S}_{l}(\alpha)$, then
	\begin{equation*}
		|a_4| \leq \dfrac{19}{36(1+3\alpha)}.
	\end{equation*}
	The result is sharp.
\end{example}
\begin{proof}
	For $f \in \mathcal{M}_{g,h}(\psi)$, we have $C_1=-1$, $C_2=1/2$, $C_3=-1/3$. Equations $\eqref{q1}$, $\eqref{q2}$ and Remark~\ref{4rem} yield $q_1=-5/2$ and $q_2=19/12$. The result follows from ~\eqref{a4a} and extremal functions $f$, up to rotations can be obtained when $f$ satisfies \[\dfrac{zf'(z)+\alpha z^2 f''(z)}{\alpha z f'(z)+(1-\alpha)f(z)}= 1-\log(1+z).\] This completes the proof.
\end{proof}
Expressing the expression $a_2 a_3-a_4$ for the function $f \in \mathcal{M}_{g,h}(\phi)$ in terms of the Schwarz function $\omega(z)= 1+\omega_1 z+\omega_2 z^2+\cdots$, we obtain the bound as follows:
$$	|a_2a_3-a_4| \leq\dfrac{B_1}{g_4-h_4} H(q_1,q_2),$$
where
\begin{equation}\label{q11}
	q_1= \dfrac{2 B_2(g_2-h_2)^2(g_3-h_3)+B_1^2(g_2-h_2)(-g_4+g_3 h_2+g_2h_3-2 h_2 h_3+h_4)}{B_1(g_2-h_2)^2(g_3-h_3)}
\end{equation}
and
\begin{align}\label{q21}
	q_2=&\dfrac{1}{B_1(g_2-h_2)^2(g_3-h_3)}\bigg(B_3(g_2-h_2)^2(g_3-h_3)+B_1B_2(g_2-h_2)(-g_4+g_3 h_2+g_2h_3\nonumber \\ &-2 h_2 h_3+h_4)+B_1^3h_2(-g_4+g_2 h_3 -h_2 h_3+h_4)\bigg).
\end{align}
\begin{example}\label{exquan}
	Let $f \in \mathcal{S}_{l}(\alpha)$. Then
	\begin{equation*}
		|a_2a_3-a_4| \leq \dfrac{1}{3(1+3\alpha)}.
	\end{equation*}
	The result is sharp.
\end{example}
\begin{proof}
	Here, we have $\Phi(z)=1-\log(1+z)$, $C_1=-1$, $C_2=1/2$ and $C_3=-1/3$. Upon replacing each $B_i$ by $(-1)^iC_i$ in~$\eqref{q11}$ and~$\eqref{q21}$, we get
	\begin{equation*}
		q_1=-\dfrac{5 \alpha^2+3 \alpha+1}{(1+\alpha)(1+2 \alpha)} \text{ and } q_2=\dfrac{19 \alpha^2-12 \alpha-4}{6(1+\alpha)(1+2 \alpha)}.
	\end{equation*}
	Here we observe that $q_1$ and $q_2$ belong to $D_2$, which is given in \cite[Lemma 3]{pokhrov}. Therefore the extremal functions $f$, up to rotations can be obtained when $f$ satisfies \[\dfrac{zf'(z)+\alpha z^2 f''(z)}{\alpha z f'(z)+(1-\alpha)f(z)}= 1-\log(1+z^3).\]
	Thus the desired result follows now.
\end{proof}
\begin{thm}\label{thma5}
	Let $f \in \mathcal{S}_{l}(\alpha)$. Then, we have \begin{equation*}
		|a_5| \leq \dfrac{107}{288(1+4\alpha)}.
	\end{equation*}
	The result is sharp.
\end{thm}
\begin{proof}
	The equations~$\eqref{gi}$,~$\eqref{gh}$,~$\eqref{omega}$ and~$\eqref{main}$ with $\psi(z)$ in place of $\phi(z)$, yield $a_5$ in terms of $p_1$, $p_2$, $p_3$ and $p_4$ as follows
	\begin{align*}
		|a_5|&=\dfrac{1}{4(1+4\alpha)}\left|\dfrac{695}{1152}p_1^4-\dfrac{27}{16}p_1^2p_2+\dfrac{1}{2}p_2^2+\dfrac{13}{12}p_1p_3-\dfrac{1}{2}p_4\right|\\& =:\dfrac{1}{8(1+4\alpha)}\left|P+\dfrac{1}{6}p_1 Q-\dfrac{1}{24}p_1^2 R\right|\\&\leq \dfrac{1}{8(1+4\alpha)}\left(|P|+\dfrac{1}{6}|p_1| |Q|+\dfrac{1}{24}|p_1|^2 |R|\right),
	\end{align*}
	where $P=p_1^4-3p_1^2p_2+p_2^2+2p_1p_3-p_4$, $Q=p_3-2p_1p_2+p_1^3$ and $R=p_2-(23/24)p_1^2$. Since $|P|$, $|Q| \leq 2$ from~$\eqref{p31}$ and $|R| \leq 2$ from~$\eqref{C}$, we get
	\begin{align*}
		|a_5|&\leq \dfrac{1}{8(1+4\alpha)}\left(2+\dfrac{2}{3}+\dfrac{1}{12}|p_1|^2-\dfrac{1}{576}|p_1|^4\right).
	\end{align*}
	Let us assume $G(p_1):=|p_1|^2/12-|p_1|^4/576$. Then, the formula given in~$\eqref{ABC1}$ yields the bound when $A=-1/576$, $B=1/12$ and $C=0$. Letting $p_1=1$, $p_2=2$, $p_3=-2/3$ and $p_4=-1/64$ shows that the result is sharp.
\end{proof}
Recall that $|H_3(1)|\leq |a_3||a_2a_4-a_3^2|+|a_4||a_4-a_2a_3|+|a_5||a_3-a_2^2|.$ Using Examples~\ref{exfek},~\ref{exsec}, ~\ref{exfourth},~\ref{exquan} and Theorem~\ref{thma5}, we can estimate the bound for $H_3(1)$ for the class $\mathcal{S}_{l}(\alpha)$, which is stated below in the following theorem:
\begin{thm}
	Let $f\in \mathcal{S}_{l}(\alpha)$. Then
	\begin{equation*}
		|H_3(1)| \leq g(\alpha),
	\end{equation*}
	where
	\begin{equation*}
		g(\alpha)= \dfrac{949 + 11388\alpha + 52493\alpha^2 + 114974\alpha^3 + 117180\alpha^4 + 42568 \alpha^5}{1728(1+4\alpha)(1+3\alpha)^2(1+2a)^4},
	\end{equation*}
	when $0\leq \alpha\leq\tfrac{2+\sqrt{15}}{11}$ and
	\begin{align*}
		g(\alpha)=&\dfrac{1}{1728(1+\alpha)(1+4\alpha)(1+3\alpha)^2(1+2\alpha)^3(61\alpha^2-20\alpha-5)}\bigg(-5069-76035\alpha- 385994\alpha^2\\& - 619570 \alpha^3 + 831511 \alpha^4 +3545777 \alpha^5 + 3327024 \alpha^6 +
		1298324 \alpha^7\bigg),
	\end{align*}
	when $\tfrac{2+\sqrt{15}}{11}\leq\alpha\leq1$.
\end{thm}
\begin{rem}
	Taking $\alpha=0$ and $1$, we get all the above  bounds for the classes $\mathcal{S}^*_{l}$ and $\mathcal{C}_{l}$, respectively.
\end{rem}
On the similar lines of the estimation of Third Hankel determinant for functions in $\mathcal{SL}^*$ in~\cite{banga2}, we compute the same for $f \in \mathcal{S}_l^*$.
\begin{thm}
	Let $f \in \mathcal{S}_l^*$, then
	$$|H_3(1)| \leq 1/9.$$ The result is sharp.\end{thm}
\begin{proof}
	The proof is on the similar lines of the proof of~\cite[Theorem~2.1]{banga2}, however the computation involves altogether new values.  Let $$\tilde{f}(z)=  z \exp\left(\int_{0}^{z}\dfrac{-\log(1+t^3)}{t}~dt\right) = z-\dfrac{z^4}{3}+\cdots,$$  clearly which belongs to $\mathcal{S}_l^*$.  The equality holds for the above defined function $\tilde{f}$, as $a_2=a_3=a_5=0$ and $a_4=-1/3$.
\end{proof}
As we know the function $z$ is univalent in $\mathbb{D}$, we have $z^n \prec z$, which further implies $1+z^n \prec 1+z$ $(n \geq 1).$ Now, there exists a Schwarz function $\omega(z)$ such that $1+z^n =1+\omega(z).$ Since $|z|<1$ and $|\omega(z)|<1$,  we can view  $1+z$ and $1+\omega(z)$  as a shifted unit disk. Thus the branch of the log function is well defined and we can write:
\begin{equation*}
	1-\log(1+z^n)=1-\log(1+\omega(z)).
\end{equation*}
Hence $1-\log(1+z^n) \prec 1-\log(1+z)$ for all $n \geq 1.$
Let us define a function $f_n$ in the class $\mathcal{A}$ as:
\begin{equation*}
	f_n(z)=z+a_{2,n}z^2+a_{3,n}z^3+\cdots=z+\sum_{m=2}^{\infty}a_{m,n}z^m.
\end{equation*}We consider the subclass $\mathcal{S}_{l,n}^*$ of $\mathcal{S}^*_{l}$ consisting of the functions $f_n$ satisfying
\begin{equation*}
	\dfrac{zf_n'(z)}{f_n(z)}=1-\log(1+z^n)\quad (n\geq 1),
\end{equation*}
which upon simplification yields
\begin{equation*}
	zf'_n(z)=f_n(z)(1-\log(1+z^n)).
\end{equation*}
Further, we have
\begin{equation*}
	\sum_{j=1}^{\infty}\left(\sum_{k=1}^{\infty}(-1)^k\dfrac{a_j}{k}z^{nk+j}\right)=\sum_{s=1}^{\infty}(s-1)a_sz^s.
\end{equation*}
On comparing the coefficients of like power terms on either side of the above equation, we get a special pattern due to which we conjecture the following:
\begin{conjecture} 	Let $f_n \in \mathcal{S}_{l,n}^*$. Then
	\[|a_{m,n}| \leq |a_{m,1}|.\]\end{conjecture}
\section{Further results}
\label{sec:5}
We recall that the set $\mathcal{B}$ is the space of all Bloch functions. An analytic function $f$ is said to be a Bloch function if it satisfies
\begin{equation}\label{sup}
	\kappa_{\mathcal{B}}(f)=\sup_{z \in \mathbb{D}} (1-|z|^2)|f'(z)| < \infty.
\end{equation}
Also, $\mathcal{B}$ is a Banach space with the norm $||.||_{\mathcal{B}}$ defined by
\begin{equation}\label{def}
	||f||_{\mathcal{B}}=|f(0)|+\kappa_{\mathcal{B}}(f), \text{ } f \in \mathcal{B}.
\end{equation}
Now, we give below a result involving Bloch function norm for the functions in the class $\mathcal{S}^*_{l}$:
\begin{thm}
	The set $\mathcal{S}^*_{l} \subseteq \mathcal{B}.$ Further, if $f \in \mathcal{S}^*_{l}$, then $||f||_{\mathcal{B}}\leq x\approx 1.27429$.
\end{thm}
\begin{proof}
	If $f \in \mathcal{S}^*_{l}$, then $z f'(z)/f(z)=1-\log(1+\omega(z)) =:g(z)$. By the structural formula, given in~$\eqref{nstr}$, we get
	\begin{equation*}
		f(z)=z \exp\left(\int_{0}^{z}\dfrac{g(t)-1}{t}dt\right).
	\end{equation*}
	Upon differentiating $f$ and further considering the modulus, we obtain
	\begin{align}\label{norm}
		|f'(z)|=&|g(z)|\left|\exp\int_{0}^{z}\dfrac{g(t)-1}{t}dt\right|\nonumber\\ \leq&  |1-\log(1+\omega(z))|\exp\left(\int_{0}^{z}\dfrac{|\log(1+\omega(t))|}{|t|}dt\right).
	\end{align}
	Let $t=r e^{i \theta_1}$ and $\omega(t)=R e^{i \theta_2}$, where $R\leq r=|t|<1$ and $-\pi<\theta_1,\theta_2<\pi$. Now, by using the similar analysis carried out in the proof of part (v) of Theorem~\ref{rad} and in Theorem~\ref{maj}, equation~$\eqref{norm}$ reduces to
	\[ f'(z)| \leq (1-\log(1-R)) \exp\left(|\log(1-R)|\int_{-\pi}^{\pi}e^{i \theta_1} d\theta_1\right) \leq 1-\log(1-r).\]
	Thus we have $g(r):=(1-|z|^2)|f'(z)|\leq (1-r^2)(1-\log(1-r)),$
	which upon differentiation, gives $	g'(r) = 1 - r + 2 r \log(1 - r).$ Taking $g'(r)=0$, yields $r_0 \approx 0.453105$. Now $g''(r_0)<0$, yields $\max_{0\leq r<1}g(r)=g(r_0 )\approx 1.27429 < \infty
	.$ Using~$\eqref{sup}$, we obtain $\mathcal{S}^*_{l} \subseteq \mathcal{B}.$ We can now, estimate the norm $||f||_{\mathcal{B}}$ for the functions in the class $\mathcal{S}^*_{l}$. Now, by using the definition of norm, given in~$\eqref{def}$, we have
	\begin{equation*}
		||f||_{\mathcal{B}}\leq f(0) + 1.27429.
	\end{equation*}
	By the normalization of the function $f,$ the result follows now  at once.\end{proof}
The following theorem gives the sufficient condition for the given function $g$ to belong to the class $\mathcal{S}^*_{l}$.
\begin{thm}
	Let $m$, $n\geq 1$ and $0\leq \lambda\leq1$. Then, $g(z)=z\exp(\alpha) \in \mathcal{S}^*_{l}$, where \begin{equation*}\alpha= \sum_{k=1}^{\infty}\dfrac{1}{k^2}\Bigg(\lambda\bigg(\dfrac{(-z)^{nk}}{n}-\dfrac{(-z)^{mk}}{m}\bigg)+\dfrac{(-z)^{mk}}{m}\Bigg).\end{equation*}
\end{thm}
\begin{proof}
	For the given $\alpha$, we have
	\begin{align*}
		g(z)=& z\Bigg(\exp\Bigg( \dfrac{(\lambda-1)}{m}z^m\left(1-\dfrac{z^m}{4}+\dfrac{z^{2m}}{9}-\cdots\right)-\dfrac{\lambda}{n}z^n\left(1-\dfrac{z^n}{4}+\dfrac{z^{2n}}{9}-\cdots\bigg)\Bigg)\right).
	\end{align*}
	Then, we have
	\begin{align*}
		\dfrac{z g'(z)}{g(z)}&= 1+(\lambda-1)z^{m}\left(1-\dfrac{z^m}{2}+\dfrac{z^{2m}}{3}-\cdots\right)-\lambda z^{n}\left(1-\dfrac{z^n}{2}+\dfrac{z^{2n}}{3}-\cdots\right)\\&= \lambda(1-\log(1+z^n))+(1-\lambda)(1-\log(1+z^m)).
	\end{align*}
	We observe that $1-\log(1+z^t)\prec 1-\log(1+z)=:\psi(z)$ for all $t\geq 1$ and the function $\psi$ is convex in $|z|<1$. Thus the result follows at once when $0 \leq \lambda \leq1$.
\end{proof}
When $m=n$, the above Theorem yields the following result:
\begin{cor}
	Let $n\geq 1$ and $g(z)=z \exp(\alpha)$, where \[\alpha=\dfrac{1}{n}\left(\sum_{k=1}^{\infty}\dfrac{(-z)^{nk}}{k^2}\right).\] We have $g \in \mathcal{S}^*_{l}$.
\end{cor}
\begin{thm}
	The class $\mathcal{S}^*_{l}$ is not a vector space.
\end{thm}
\begin{proof}
	For if, the class $\mathcal{S}^*_{l}$ is a vector space, then the class preserves additive property, that is, whenever two functions belong to the class $\mathcal{S}^*_{l}$, then their sum also belongs to the class $\mathcal{S}^*_{l}$.
	Let $f_1$ and $f_2 \in \mathcal{S}^*_{l}$. Then, using~$\eqref{nstr}$, we obtain
	\begin{equation*}
		f_1(z)=z \exp\left(\int_{0}^{z}\dfrac{-\log(1+\omega_1(t))}{t}dt\right)\text{ and } f_2(z)=z \exp\left(\int_{0}^{z}\dfrac{-\log(1+\omega_2(t))}{t}dt\right),\end{equation*}
	for some Schwarz functions $\omega_1$ and $\omega_2$.
	Thus, the sum of the functions, $f_1+f_2$ to be in $\mathcal{S}^*_{l}$, there should exist some Schwarz function $\omega(z)$ such that
	\begin{equation*}
		\omega(z)= \dfrac{\exp(-z(A'(z)\exp A(z)+B'(z)\exp B(z)))}{\exp A(z) +\exp B(z)}-1,
	\end{equation*}
	where
	\begin{equation*}
		A(z)= \int_{0}^{z}\dfrac{-\log(1+\omega_1(t))}{t}dt \text{ and } B(z)=\int_{0}^{z}\dfrac{-\log(1+\omega_2(t))}{t}dt.
	\end{equation*}
	Then $\omega(0)=0$ and $|\omega(z)|<1$ for all $A$ and $B$. Let $\omega_1(z)= z$ and $\omega_2(z)=z^2$. We observe that $\omega(z) \approx 1.03053$ at $z=-(\tfrac{1}{2}+i\tfrac{2}{3})$, which contradicts the existence of Schwarz function $\omega(z)$ with $|\omega(z)| <1.$  Hence the assertion follows. \end{proof}
The following theorem is an immediate consequence of the growth Theorem of $\mathcal{S}^*(\Phi)$.
\begin{thm}
	Let $f \in \mathcal{S}^*_{l}$, then we have
	\begin{equation*}
		|f(z)| \leq |z| \exp\left(\sum_{n=1}^{\infty}\dfrac{(-1)^{n+1}}{n^2}\right)=|z|L\text{ }\text{ }(z \in \mathbb{D}),
	\end{equation*}
	where $L\approx 0.822467.$
\end{thm}
\begin{proof}
	In view of Remark~\ref{distrem}, we get
	\begin{equation*}
		t_{\Psi}(r) \leq |f(z)| \leq -t_{\Psi}(-r),
	\end{equation*}
	For $|z|=r$, we have
	\begin{equation*}
		\log \left|\dfrac{f(z)}{z}\right| \leq \int_{0}^{r}\dfrac{\log(1+t)}{t}dt\leq \int_{0}^{1}\dfrac{\log(1+t)}{t} dt=\sum_{n=1}^{\infty}\dfrac{(-1)^{n+1}}{n^2}.
	\end{equation*}
	The convergent nature of the series on the right side of the above equality yields the desired result at once.
\end{proof}
\noindent\textbf{Concluding Remark:} The classes $\mathcal{S}^*(\Phi),$  $\mathcal{C}(\Phi)$, $\mathcal{S}^*_{l}$, $\mathcal{C}_{l}$, $\mathcal{M}_{g,h}(\phi)$ and $\mathcal{M}_\alpha(\Phi)$ studied here are all special cases of $\mathcal{A}(g,h,\varphi)$ and it leaves ample scope open for further studies in specializing the class for different choices of $g$ and $h$ together with altered conditions on $\varphi$. Also, the idea of non-Ma-Minda and a special type of Ma-Minda introduced here can be used to define new classes and studied in the direction pointed out here.	
	
	\noindent \textbf{Acknowledgements.} The work presented here was supported by a Research Fellowship from the Department of Science and Technology, New Delhi.
	
\end{document}